\documentclass[final,twoside,11pt]{entics}
\usepackage{enticsmacro}
\usepackage{graphicx}
\usepackage[all]{xy}
\usepackage{float}
\def\conf{ISDT 2022} 	
\volume{2}			
\def\volu{2}			
\def\papno{12}			
\def\mydoi{{\normalshape  \href{https://doi.org/10.46298/entics.10312}{10.46298/entics.10312}}}
\def\copynames{Y. Song, J. Yang} 
\def\runtitle{Strongly Continuous Domains}

\def\lastname{Song and Yang}
\begin{document}

\begin{frontmatter}
  \title{Strongly Continuous Domains\thanksref{ALL}}
  \author{Yinglong Song\thanksref{a}\thanksref{b}\thanksref{myemail}}	
   \author{Jinbo Yang\thanksref{a}\thanksref{b}\thanksref{coemail}}		
   \address[a]{School of Mathematics and Statistics, Jiangxi Normal University,				
    Nanchang, Jiangxi 330022, People's Republic of China}
   \address[b]{Jiangxi Provincial Center for Applied Mathematics, Jiangxi Normal University, Nanchang, Jiangxi 330022, People's Republic of China}

  \thanks[ALL]{Supported by the NSF of China (12071188, 11361028, 11671008) and Science and Technology Project from Jiangxi Education Department (GJJ150344)}   
   \thanks[myemail]{Email: \href{mailto:YLSong1997@163.com} {\texttt{\normalshape
        YLSong1997@163.com}}}

  \thanks[coemail]{Email:  \href{mailto:jbyang73@163.com} {\texttt{\normalshape
        jbyang73@163.com}}}

\begin{abstract}
 Strong Scott topology introduced by X. Xu and D. Zhao is a kind of new topology which is finer than upper topology and coarser than Scott topology. Inspired by the topological characterizations of continuous domains and hypercontinuous domains, we introduce the concept of strongly continuous domains and investigate some properties of strongly continuous domains. In particular, we give the definition of strong way-below relation and obtain a characterization of strongly continuous domains via the strong way-below relation. We prove that the strong way-below relation on a strongly continuous domain satisfies the interpolation property, and clarify the relationship between strongly continuous domains and continuous domains, and the relationship between strongly continuous domains and hypercontinuous domains. We discuss the properties of strong Scott topology and strong Lawson topology, which is the common refinement of the strong Scott topology and the lower topology, on a strongly continuous domain.
\end{abstract}
\begin{keyword}
  strongly continuous domains, strong way-below relation, strong Scott topology, strong Lawson topology
\end{keyword}
\end{frontmatter}
\section{Introduction}
In non-Hausdorff topology and domain theory, $d$-spaces, well-filtered spaces and sober spaces form three of the most important classes of $T_0$ spaces. In the past few years, some remarkable progresses have been achieved in the research of these three kinds of spaces(see, e.g., \cite{XSXZ,XXZ,XZ,XZ2,XYZ}). In order to reveal finer links between well-filtered spaces and $d$-spaces, X. Xu and D. Zhao introduced another class of $T_0$ spaces--\emph{strong $d$-spaces} which lie between $d$-spaces and $T_1$ spaces in \cite{XZ}. It is well-known that a $T_0$ space $X$ is a $d$-space iff $X$ is a \emph{directed complete poset} (\emph{dcpo}, for short) and the topology on $X$ is coarser than the Scott topology on $X$ with respect to its \emph{specialization order}. In order to obtain some characterizations of strong $d$-spaces (\cite[Proposition 3.22]{XZ}), the authors \cite{XZ} introduced the concept of \emph{strong Scott topology} and showed that for a $T_0$ space $X$,  if $X$ is a strong $d$-space, then $X$ is a $d$-space and the topology on $X$ is coarser than the strong Scott topology on $X$, and the converse is true if $X$ is a sup semilattice with the specialization order.

Strong Scott topology is a kind of new topology which is finer than the upper topology and coarser than the Scott topology. Inspired by the topological characterizations of continuous domains (resp., hypercontinuous domains) that a dcpo $P$ is a continuous domain (resp., a hypercontinuous domain) if and only if $P$ is a $C$-space with respect to the Scott topology (resp., upper topology), we introduce the concept of \emph{strongly continuous domains}. That is to say, a dcpo $P$ is called a strongly continuous domain if $P$ is a $C$-space with respect to the strong Scott topology. In this paper, we will investigate the properties of strongly continuous domains from the aspects of order structure and topology structure. In particular, we introduce the notion of \emph{strong way-below relation} and obtain a characterization of strongly continuous domains via this relation. We prove that the strong way-below relation on a strongly continuous domain satisfies the interpolation property and clarify the relationship between strongly continuous domains and continuous domains, and the relationship between strongly continuous domains and hypercontinuous domains. We prove that a strongly continuous domain $P$ endowed with the strong Scott topology is a locally compact sober space, and $P$ is a $T_{2}$ space with respect to the \emph{strong Lawson topology} which is the common refinement of the strong Scott topology and the lower topology.

\section{Preliminaries}

We now recall some basic definitions and notations needed in this paper. For further details, we refer the reader to \cite{GHK,GL}.
For a set $X$, the family of all finite sets in $X$ is denoted by $X^{(<\omega)}$. For a poset $L$, if $L$ has a greatest element (resp., smallest element), it is called the unit or top element (resp., zero or bottom element) of $L$ and is written as $1$ (resp., $0$). For any $x\in L$ and $A\subseteq L$, let $\downarrow x=\{y\in L: y\leq x\}$, $\downarrow \!A=\bigcup\{\downarrow \!x: x\in A\}$; $\uparrow \!x$ and $\uparrow \!A$ are defined dually. A subset $A$ is called a \emph{lower set} (resp., an \emph{upper set}) if $A=\ \downarrow \!A$ (resp., $A=\ \uparrow \!A$). A nonempty subset $D$ of $L$ is called \emph {directed} if every two elements in $D$ have an upper bound in $D$. $L$ is called a \emph{directed complete poset}, or \emph{dcpo} for short, if every directed subset of $L$ has the least upper bound in $L$. A subset $I\subseteq L$ is called an \emph{ideal} of $L$ if $I$ is a directed and a lower set. Dually, we define the concept of \emph{filters}. $L$ is called an (inf) \emph{semilattice} if any two elements $a, b$ have an inf, denoted by $a\wedge b$. Dually, $L$ is a \emph{sup semilattice} if any two elements $a,b$ have a sup, denoted by $a\vee b$. $L$ is called a \emph{complete semilattice} if $L$ is a dcpo and every nonempty subset of $L$ has an inf. $L$ is called a \emph{complete lattice} if every subset of $L$ has a sup and an inf. The \emph{upper topology} on $L$, generated by $\{L\}\cup\{L\setminus\!\downarrow \!x: x\in L\}$ (as a subbase), is denoted by $\upsilon(L)$. Dually, we define the \emph{lower topology} on $L$ and denote it by $\omega(L)$. A subset $U$ of $L$ is called \emph{Scott open} if $U$ is an upper set and for any directed subset $D$ for which $\bigvee D$ exists, $\bigvee D\in U$ implies $D\cap U\neq\emptyset$. All Scott open subsets of $L$ form a topology, and we call this topology the \emph{Scott topology} on $L$ and denote it by $\sigma(L)$. The complement of a Scott open set is called \emph{Scott closed}.

For a topological space $X$, let $\mathcal{O}(X)$ (resp., $\mathcal{C}(X)$) be the lattice of all open subsets (closed subsets) in $X$. For $A\subseteq X$, the \emph{closure} (resp., \emph{interior}) of $A$ in $X$ is denoted by cl$_XA$ (resp., int$_XA$). For a $T_0$ space $X$, we use $\leq_{X}$ to represent the \emph{specialization order} of $X$, that is , $x\leq_{X}y$ if and only if $x\in {\rm cl}_X(\{y\})$. Unless otherwise stated, throughout the paper, whenever an order-theoretic concept is mentioned, it is to be interpreted with respect to the specialization order.

\begin{definition}{\rm(\cite{GHK})} Let $L$ be a poset.
\begin{enumerate}
\item[\rm{(1)}] We say that $x$ is \emph{way below} $y$, in symbols $x\ll y$, iff for any directed subset $D\subseteq L$ for which $\sup D$ exists, $y\leq \sup D$ implies the existence of a $d\in D$ with $x\leq d$.
\item[\rm{(2)}] $L$ is called \emph{continuous}, if for any $x\in L, \Downarrow x=\{u\in L: u\ll x\}$ is directed and $x=\bigvee\Downarrow x$. A dcpo which is continuous is called a \emph{continuous domain}, or \emph{domain}. A domain which is a complete lattice is called a \emph{continuous lattice}.
\item[\rm{(3)}] We say that $x\prec y$ iff $y\in$ int$_{\upsilon(L)}\!\uparrow \!x$. And $L$ is called \emph{hypercontinuous} if for any $x\in L, \{y\in L: y\prec x\}$ is directed and $x=\bigvee\{y\in L: y\prec x\}$. A dcpo which is hypercontinuous is called a \emph{hypercontinuous domain}. A hypercontinuous domain which is a complete lattice is called a \emph{hypercontinuous lattice}.
\end{enumerate}
\end{definition}

\begin{definition}{\rm(\cite{VP})} Let $L$ be a complete lattice. A binary relation $\lhd$ on $L$ is defined as follows: $x\lhd y$ iff for any $S\subseteq L, y\leq\bigvee S$ implies that $x\in\ \!\downarrow \!S$. $L$ is a \emph{prime continuous lattice} if for all $x\in L, x=\bigvee\{y\in L: y\lhd x\}$.
\end{definition}

\begin{definition}{\rm(\cite{ZH})}
Let $X$ be a $T_{0}$ space. $X$ is called a $C$-space if for any $U\in\mathcal{O}(X)$ and $x\in U$, there exists $y\in U$ such that $x\in\!$ int$_{X}\!\uparrow \!y\subseteq\ \!\uparrow \!y\subseteq U$.
\end{definition}

\begin{proposition}{\rm(\cite{ZH})}\label{pro-2.4}
For a $T_{0}$ space $X$, $X$ is a $C$-space if and only if $(\mathcal{O}(X),\subseteq)$ is a completely distributive lattice.
\end{proposition}

\begin{theorem}{\rm(\cite{GHK})} Let $L$ be a dcpo.
\begin{enumerate}
\item[\rm{(1)}] $L$ is a domain iff $(L,\sigma(L))$ is a $C$-space.

\item[\rm{(2)}] $L$ is a hypercontinuous domain iff $(L,\upsilon(L))$ is a $C$-space.
\end{enumerate}
\end{theorem}

\begin{theorem}\label{the-2.6}{\rm(\cite{GHK})} Let $L$ be a complete lattice. The following conditions are equivalent:
\begin{enumerate}
\item[\rm{(1)}] $L$ is completely distributive;

\item[\rm{(2)}] $L$ is distributive, and both $L$ and $L^{\rm{op}}$ are continuous lattices;

\item[\rm{(3)}] $L$ is continuous and every element is the sup of co-primes.
\end{enumerate}
\end{theorem}

\begin{definition}{\rm(\cite{XZ})} Let $L$ be a dcpo. A subset $U\subseteq L$ is called \emph{strongly Scott open} if (i) $U=\ \uparrow \!U$, and (ii) for any directed subset $D$ of $L$ and $x\in L$, $\bigcap _{d\in D}\uparrow \!d~\cap \uparrow \!x\subseteq U$ (that is $\uparrow\bigvee D\ \cap\uparrow x\subseteq U$) implies $\uparrow \!d\ \cap \uparrow \!x\subseteq U$ for some $d\in D$. Let $\sigma^{s}(L)$ denote the set of all strongly Scott open subsets of $L$.

Clearly, if $U,V\in \sigma ^{s}(L)$, then $U\cap V\in \sigma ^{s}(L)$. The topology generated by $\sigma ^{s}(L)$ (as a base) is called the \emph{strong Scott topology} on $L$ and denote it by $\sigma _{s}(L)$. The space $(L,\sigma _{s}(L))$ is called the \emph{strong Scott space} of $L$, and will be denoted by $\sum _{s}L$.

In order not to cause ambiguity, the elements in $\sigma _{s}(L)$ will be called \emph{strong Scott topology open sets}.
\end{definition}

\begin{proposition}{\rm(\cite{XZ})}
Let $L$ be a poset, then $\upsilon(L)\subseteq \sigma^{s}(L)\subseteq \sigma_{s}(L)\subseteq \sigma(L)$.
\end{proposition}

\section{Strongly continuous domains}

In this section we introduce the notion of strongly continuous domains and give some characterizations of strongly continuous domains from the aspect of order structure.

\begin{definition} \label{def-3.1}Let $L$ be a dcpo. $L$ is called a \emph{strongly continuous domain} if $(L,\sigma_{s}(L))$ is a $C$-space, i.e., for any $U\in\sigma_s(L)$ and $x\in L$, $x\in U$ implies there exists a $u\in U$ such that $x\in{\rm{int}}_{\sigma_s(L)}\!\uparrow \!u\subseteq\ \! \uparrow \!u\subseteq U$.
\end{definition}

The reason we adopt the terminology ``strongly continuous'' in the paper is that we introduce the notion of strongly continuous domains via strong Scott topology. It should be pointed out that the terminology ``strongly continuous'' in this paper is completely different from that in \cite{XM}.

\begin{definition}(\cite{XM}) Let $P$ be a poset and $x, y\in P$. We write $x\ll_ly$, if for any directed set $D$ and any upper bound $z$ of $D$ with $y\leq\sup_zD$, there is a $d\in D$ such that $x\leq d$, where $\sup_zD$ denotes the supremum of $D$ in $\downarrow \!z$. $P$ is called a strongly continuous poset if for all $x\in P$, $\Downarrow_lx=\{y\in L: y\ll_lx\}$ is directed and $\sup\Downarrow_lx=x$.
\end{definition}

\begin{definition} Let $L$ be a poset. A binary relation $\ll_{s}$ on $L$ is defined as follows: $x\ll_{s}y$ if for any directed subset $D\subseteq L$ and $a\in L$, $\bigcap _{d\in D}\!\uparrow \!d\ \cap \uparrow \!a\subseteq\ \!\uparrow \!y$ implies $\uparrow \!d\ \cap \uparrow \!a\subseteq\ \!\uparrow \!x$ for some $d\in D$.
\end{definition}

\begin{remark} Let $L$ be a dcpo and $u, x, y, z\in L$. We have the following.
\begin{enumerate}
\item[(1)] $x\ll_{s}y$ implies $x\ll y$, and $x\ll_{s}y$ is equivalent to $x\ll y$ if $L$ is a sup semilattice.
\item[(2)] $u\leq x\ll_{s}y\leq z$ implies $u\ll_{s}z$.
\item[(3)] $x\ll_{s}z$, $y\ll_{s}z$ implies $x\vee y\ll_{s}z$ whenever $x\vee y$ exists in $L$.
\item[(4)] $0\ll_{s}x$ whenever $L$ has a smallest element $0$.
\end{enumerate}
Hence $\ll_{s}$ is an auxiliary relation on $L$.
\end{remark}
For $x\in L$, we write $\Downarrow_sx=\{y\in L: y\ll_sx\}$ and $\Uparrow_sx=\{y\in L: x\ll_sy\}$.

\begin{proposition}
Let $L$ be a dcpo. Suppose that there exists a directed set $D\subseteq\ \Downarrow_{s}\!x$ such that $\sup D=x$, then we have the following conditions.
\begin{enumerate}
\item[\rm{(1)}] $\Downarrow_{s}\!x$ is directed and $\sup\Downarrow_{s}\!x=x$;
\item[\rm{(2)}] If $y\ll_{s}x$ in $\downarrow \!x$, then $y\ll_{s}x$ in $L$.
\end{enumerate}
\end{proposition}

\begin{proof}
(1): Let $y_{1},y_{2}\in\ \!\Downarrow_{s}\!\!x$, i.e., $y_{1}\ll_{s}x$ and $y_{2}\ll_{s}x$. Since $\ll_{s}\ \!\subseteq\ \!\ll$, then $y_{1}\ll x$ and $y_{2}\ll x$. Since $\textrm{sup}D=x$, then $y_{1}\leq d_{1}, y_{2}\leq d_{2}$ for some $d_{1},d_{2}\in D$. Since $D$ is directed, there exists $d_{3}\in D$ such that $y_{1}\leq d_{1}\leq d_{3}$, $y_{2}\leq d_{2}\leq d_{3}$, then $d_{3}\geq y_{1},y_{2}$. Thus $\Downarrow_{s}\!x$ is directed by $d_{3}\in D\subseteq\ \!\Downarrow_{s}\!x$. Since $x=\textrm{sup}D\leq \textrm{sup}\Downarrow_{s}\!x\leq \textrm{sup}\downarrow \!x=x$, we have $\textrm{sup}\Downarrow_{s}\!x=x$.\\
(2): Let $y\ll_{s}x$ in $\downarrow \!x$. Then $y\ll x$. Since directed set $D\subseteq\ \!\Downarrow_{s}\!x$ and $\textrm{sup}D=x$, $y\leq d$ for some $d\in D$. Therefore $y\ll_{s}x$ by $d\in\ \!\Downarrow_{s}\!x$.
\end{proof}

\begin{theorem}\label{The-3.5} Let $L$ be a dcpo. The following two conditions are equivalent:
\begin{enumerate}
\item[\rm{(1)}] $L$ is a strongly continuous domain;
\item[\rm{(2)}] For any $x\in L$, $\Downarrow_{s}\!x$ is directed, $x=\sup\Downarrow_{s}\!x$ and $\Uparrow_{s}\!x\in\sigma_{s}(L)$.
\end{enumerate}
\end{theorem}

\begin{proof}
$(1)\Rightarrow(2)$: Suppose that $L$ is a strongly continuous domain. Then $(L,\sigma_{s}(L))$ is a $C$-space. For any $a\in L$, let $H_{a}=\{x\in L:a\in \textrm{int}_{\sigma_{s}(L)}(\uparrow \!\!x)\}$. Then since $(L,\sigma_{s}(L))$ is a $C$-space and $a\in L\in\sigma_{s}(L)$, there exists $x\in L$ such that $a\in \textrm{int}_{\sigma_{s}(L)}(\uparrow \!\!x)\subseteq\ \!\uparrow \!\!x\subseteq L$. Hence $H_{a}\neq\emptyset$. For any $x_{1},x_{2}\in H_{a}$, since $(L,\sigma_{s}(L))$ is a $C$-space, there exists $x_{3}\in L$ such that $a\in\textrm{int}_{\sigma_{s}(L)}(\uparrow \!\!x_{3})\subseteq\ \!\uparrow\!\! x_{3}\subseteq(\textrm{int}_{\sigma_{s}(L)}(\uparrow \!\!x_{1}))\cap(\textrm{int}_{\sigma_{s}(L)}(\uparrow \!\!x_{2}))$, hence $x_{3}\in H_{a}$ and $x_{1},x_{2}\leq x_{3}$, thus $H_{a}$ is directed. If $x\in H_{a}$ and $y\leq x$, then $a\in\textrm{int}_{\sigma_{s}(L)}(\uparrow \!\!x)\subseteq\textrm{int}_{\sigma_{s}(L)}(\uparrow \!\!y)$, thus $y\in H_{a}$, hence $H_{a}$ is a lower set. We now show that $\bigvee H_{a}=a$. Obviously, $a$ is an upper bound of $H_{a}$. Let $b$ be an other upper bound of $H_{a}$, i.e., $H_{a}\subseteq\ \!\downarrow \!b$. Assume that $a\nleq b$, then $a\in L\setminus\downarrow \!b\in\upsilon(L)\subseteq\sigma_{s}(L)$, since $(L,\sigma_{s}(L))$ is a $C$-space, there exists $x\in L\setminus\downarrow \!b$ such that $a\in\textrm{int}_{\sigma_{s}(L)}(\uparrow \!\!x)\subseteq\ \!\uparrow\!\! x\subseteq L\setminus\downarrow \!b$. Hence $x\in H_{a}$ and $x\nleq b$, this is a contradiction. Thus $\bigvee H_{a}=a$. We now show that $H_{a}=\ \!\Downarrow_{s}\!a$. On the one hand, if $x\in H_{a}$, then there exists $U\in\sigma^{s}(L)$ such that $a\in U\subseteq\textrm{int}_{\sigma_{s}(L)}(\uparrow \!\!x)$. For any directed subset $D\subseteq L$ and $z\in L$, if $\uparrow\!\bigvee D~\cap\uparrow \!z\subseteq\ \!\uparrow \!a$, then $\uparrow\!\bigvee D\ \!\cap\uparrow \!z\subseteq\ \!\uparrow \!a\subseteq U\in\sigma^{s}(L)$, hence $\uparrow \!d~\cap\uparrow \!z\subseteq U\subseteq\textrm{int}_{\sigma_{s}(L)}(\uparrow \!\!x)\subseteq\ \!\uparrow \!x$ for some $d\in D$. Thus $x\ll_{s}a$. Therefore $H_{a}\subseteq\ \!\Downarrow_{s}\!a$. On the other hand, for any $y\in\ \!\Downarrow_{s}\!a$, i.e., $y\ll_{s}a=\bigvee H_{a}$, then $y\ll a=\bigvee H_{a}$. Since $H_{a}$ is an ideal, $y\in\ \!\downarrow \!H_{a}=H_{a}$. Now we prove that $\Uparrow_{s}\!\!a\in\sigma_{s}(L)$. Since $x\in\ \!\Uparrow_{s}\!\!a$ iff $a\in\ \!\Downarrow_{s}\!x=H_{x}$ iff $x\in\textrm{int}_{\sigma_{s}(L)}(\uparrow \!\!a)$, we have $\Uparrow_{s}\!a=\textrm{int}_{\sigma_{s}(L)}(\uparrow \!a)\in\sigma_{s}(L)$.\\
$(2)\Rightarrow(1)$: Suppose $a,b\in L$ such that $a\ll_{s}b$. Since $\Uparrow_{s}\!\!a\in\sigma_{s}(L)$ and $\Uparrow_{s}\!\!a\subseteq\ \!\uparrow \!\!a$, then we have $b\in\ \!\Uparrow_{s}\!a\subseteq\textrm{int}_{\sigma_{s}(L)}(\uparrow \!\!a)$. We now show that $(L,\sigma_{s}(L))$ is a $C$-space. Suppose that $U\in\sigma_{s}(L)$ and $b\in U$, then $\bigvee\{x\in L:x\ll_{s}b\}=b\in U$. Since $\sigma_{s}(L)\subseteq\sigma(L)$, $\bigvee\{x\in L:x\ll_{s}b\}=b\in U\in\sigma_{s}(L)\subseteq\sigma(L)$, thus $\{x\in L:x\ll_{s}b\}\cap U\neq\emptyset$, hence $x\in U$ for some $x\in\ \!\Downarrow_{s}\!b$. Thus $b\in\textrm{int}_{\sigma_{s}(L)}(\uparrow \!x)$. Therefore $(L,\sigma_{s}(L))$ is a $C$-space.
\end{proof}

According to the Definition \ref{def-3.1} and Theorem \ref{The-3.5}, strongly continuous domains can be characterized by means of strong way-below relation.

\begin{proposition}
Let $L$ be a strongly continuous domain. If $x\ll_{s}z$ and $z\leq\bigvee D$ for a directed set in $L$, then $x\ll_{s}d$ for some $d\in D$.
\end{proposition}
\begin{proof}
Let $I=\bigcup\{\Downarrow_{s}\!d:d\in D\}$. Since $L$ is strongly continuous, $\bigvee I=\bigvee D$ and $I$ is an ideal. If $x\ll_{s}z$ and $z\leq\bigvee D$, then $x\ll z\leq\bigvee D=\bigvee I$. Since strongly continuous domains are continuous domains by definition, $x\in I$. Thus $x\ll_{s}d$ for some $d\in D$.
\end{proof}

\begin{proposition}\label{pro-3.7}
In a strongly continuous domain $L$, $\ll_{s}$ satisfies the interpolation property, i.e., for all $x, y\in L$, $x\ll_sy$ implies that there exists a $z\in L$ such that $x\ll_sz\ll_sy$.
\end{proposition}

\begin{proof}
Suppose $x,y\in L$ such that $x\ll_{s}y$, we now show that $x\ll_{s}z\ll_{s}y$ for some $z\in L$. Let $D=\bigcup_{a\in\Downarrow_{s}y}\Downarrow_{s}\!a=\{b\in L:$ there is an $a\in L$ with $b\ll_{s}a\ll_{s}y\}$. We claim that $D$ is directed. For any $b_{1},b_{2}\in D$, we have $b_{1}\ll_{s}a_{1}\ll_{s}y$ for some $a_{1}\in L$ and $b_{2}\ll_{s}a_{2}\ll_{s}y$ for some $a_{2}\in L$. Since $\Downarrow_{s}\!y$ is directed, there exists $a_{3}\in\ \!\Downarrow_{s}\!y$ such that $a_{1},a_{2}\leq a_{3}$. Hence $b_{1},b_{2}\ll_{s}a_{3}$. Since $\Downarrow_{s}\!a_{3}$ is directed, there exists $b_{3}\in\ \!\Downarrow_{s}\!a_{3}$ such that $b_{1},b_{2}\leq b_{3}\ll_{s}a_{3}\ll_{s}y$. Thus $b_{3}\in \!D$ and $b_{1},b_{2}\leq b_{3}$. Therefore $D$ is directed. Since $\bigvee D=\bigvee\{\bigvee\Downarrow_{s}\!a:a\in\ \!\Downarrow_{s}\!y\}=\bigvee\{a:a\in\ \!\Downarrow_{s}\!y\}=\bigvee\Downarrow_{s}\!y=y$, $\bigcap_{d\in D}\!\uparrow \!d\subseteq\ \!\uparrow \!y$. It follows from $x\ll_{s}y$ that $\uparrow \!d\subseteq\ \!\uparrow \!x$ for some $d\in \!D$. Thus there exists $a\in L$ such that $d\ll_{s}a\ll_{s}y$. Therefore $x\ll_{s}a\ll_{s}y$.
\end{proof}

Recall that hypercontinuous domains are continuous domains. Now we discuss the relationship between strongly continuous domains and hypercontinuous domains, and the relationship between strongly continuous domains and continuous domains.

\begin{lemma}{\rm(\cite{GHK})} \label{lem3.8}Let $L$ be a dcpo. Then the following conditions are equivalent:
\begin{enumerate}
\item[\rm{(1)}] $L$ is hypercontinuous;
\item[\rm{(2)}] $L$ is continuous and $\ll=\prec$;
\item[\rm{(3)}] $L$ is continuous and $\upsilon(L)=\sigma(L)$.
\end{enumerate}
\end{lemma}

\begin{theorem}\label{the3.9} Let $L$ be a dcpo. Then the following conditions are equivalent:
\begin{enumerate}
\item[\rm{(1)}] $L$ is hypercontinuous;
\item[\rm{(2)}] $L$ is strongly continuous and $\ll_{s}=\prec$;
\item[\rm{(3)}] $L$ is strongly continuous and $\upsilon(L)=\sigma_{s}(L)$.
\end{enumerate}
\end{theorem}
\begin{proof}
$(1)\Rightarrow(2)$: Let $L$ be a hypercontinuous domain, thus $\ll=\prec$ by Lemma \ref{lem3.8}. Since $\prec\subseteq\ll_{s}\subseteq\ll$, we have $\prec=\ll_{s}$. For any $x\in L$, $\Uparrow_{s}\!x=\{y\in L:x\prec y\}=\textrm{int}_{\upsilon(L)}\!\uparrow \!x\in\upsilon(L)\subseteq\sigma_{s}(L)$.\\
$(2)\Rightarrow(3)$: Obviously, $\upsilon(L)\subseteq\sigma_{s}(L)$. For each $U\in\sigma_{s}(L)$, we show that $U\in\upsilon(L)$. For any $x\in U$, $x=\bigvee^{\uparrow}\!\Downarrow_{s}\!x=\bigvee^{\uparrow}\{y\in L:y\prec x\}\in U\in\sigma_{s}(L)\subseteq\sigma(L)$, hence there is a $y\prec x$ with $y\in U$. Since $y\prec x$, $x\in \textrm{int}_{\upsilon(L)}\!\uparrow \!y$. Then $x\in \textrm{int}_{\upsilon(L)}\!\uparrow \!y\subseteq\ \!\uparrow \!y\subseteq U$. Thus $U\in\upsilon(L)$. Therefore $\upsilon(L)=\sigma_{s}(L)$.\\
$(3)\Rightarrow(1)$: $L$ is strongly continuous iff $(L,\sigma_{s}(L))$ is a $C$-space iff $(L,\upsilon(L))$ is a $C$-space iff $L$ is hypercontinuous.
\end{proof}

\begin{example}\label{example3.10} Let $L=\{0\}\cup\{a_{1},a_{2},\cdots,a_{n},\cdots\}$ with the order generated by
\begin{enumerate}
\item[\rm{(a)}] $a_{n}>0$ for all $n\in \mathbb{N}$;
\item[\rm{(b)}] There is no order relationship between $a_{i}$ and $a_{j}$ for all $i,j\in\mathbb{N}$ (Fig. \ref{Fig.1}).
\end{enumerate}
Then $L$ is a strongly continuous domain, but not a hypercontinuous domain.
\end{example}

\begin{figure}[H]
  \centering
  \includegraphics[width=0.45\textwidth]{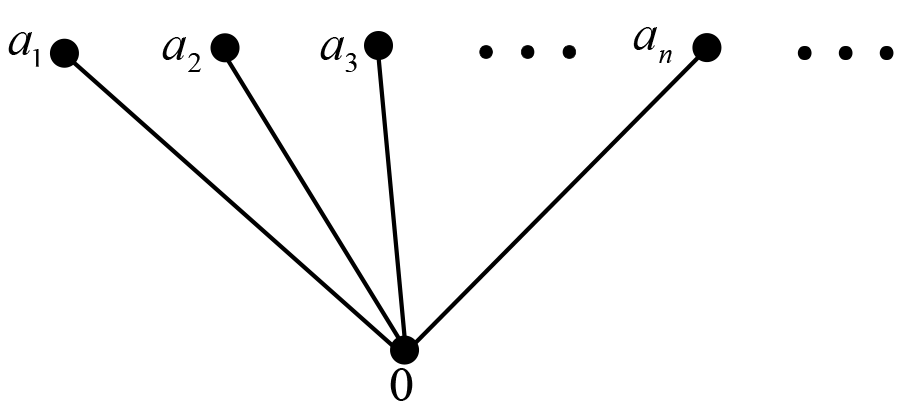}\\
  \caption{The poset $L$.}
  \label{Fig.1}
\end{figure}

\begin{proof}
Since the directed subset in $L$ can only be a single point set or a binary set $\{0,a_{n}\}(n\in\mathbb{N})$, $L$ is a dcpo. It is easy to show that $L$ is not a hypercontinous domain. Now we prove that $L$ is a strongly continuous domain. Let $x\in L$, we now show that $\Downarrow_{s}\!x$ is directed. Obviously, $0\ll_{s}0$, then $\Downarrow_{s}\!0=\{0\}$. For each $n\in\mathbb{N}$, let directed set $D\subseteq L$ and $z\in L$ such that $\bigcap_{d\in D}\uparrow \!d\ \!\cap\uparrow \!z\subseteq\ \!\uparrow \!a_{n}=\{a_{n}\}$. Then there are only the following cases: (\textrm{i}) $D=\{0\}$, $z=\{a_{n}\}$; (ii) $D=\{a_{n}\}$, $z\in L$; (iii) $D=\{a_{m}\}(m\neq n)$, $z\neq\{0\}$ and $z\neq\{a_{m}\}$; (iv) $D=\{0,a_{n}\}$, $z\in L$; (v) $D=\{0,a_{m}\}(m\neq n)$, $z\neq\{0\}$ and $z\neq\{a_{m}\}$. Clearly, there exists $d\in D$ such that $\uparrow \!d\ \!\cap\uparrow \!z\subseteq\ \!\uparrow \!a_{n}=\{a_{n}\}$ in each case. Hence $a_{n}\ll_{s}a_{n}$. therefore $\Downarrow_{s}\!a_{n}=\{0,a_{n}\}$. Thus $\Downarrow_{s}\!x$ is directed and $x=\bigvee\!\Downarrow_{s}\!x$ for all $x\in L$. Now we prove that $\Uparrow_{s}\!x\in\sigma_{s}(L)$. Clearly, $\Uparrow_{s}\!0=L\in\sigma_{s}(L)$. For each $n\in\mathbb{N}$, $\Uparrow_{s}\!a_{n}=\{a_{n}\}$ and $\Uparrow_{s}\!a_{n}\in\sigma^{s}(L)\subseteq\sigma_{s}(L)$ according to the above proof. Therefore $L$ is a strongly continuous domain.
\end{proof}

According to the Theorem \ref{the3.9} and Example \ref{example3.10}, a hypercontinuous domain is a strongly continuous domain, but the converse is not true.

\begin{theorem}\label{the3.11} Let $L$ be a dcpo. Then the following conditions are equivalent:
\begin{enumerate}
\item[\rm{(1)}] $L$ is strongly continuous;
\item[\rm{(2)}] $L$ is continuous, $\ll_{s}=\ll$ and $\Uparrow_{s}\!x\in\sigma_{s}(L)$ for all $x\in L$;
\item[\rm{(3)}] $L$ is continuous and $\sigma_{s}(L)=\sigma(L)$.
\end{enumerate}
\end{theorem}
\begin{proof}
$(1)\Rightarrow(2)$: Since $L$ is strongly continuous and $\ll_{s}\subseteq\ll$, $L$ is continuous. If $x\ll y=\bigvee^{\uparrow}\{z\in L:z\ll_{s}y\}$, then there exists $z\in\ \!\Downarrow_{s}\!y$ such that $x\leq z$, hence $x\ll_{s}y$. Clearly, $\Uparrow_{s}\!x\in\sigma_{s}(L)$ for all $x\in L$.\\
$(2)\Rightarrow(3)$: We only need prove that $\sigma(L)\subseteq\sigma_{s}(L)$. Let $U\in\sigma(L)$, for each $x\in U$, $x=\bigvee^{\uparrow}\!\Downarrow \!x=\bigvee^{\uparrow}\!\Downarrow_{s}\!x\in U\in\sigma(L)$, then there exists $y\in\ \!\Downarrow_{s}\!x$ such that $y\in U$. Thus $x\in\ \!\Uparrow_{s}\!y\subseteq\ \!\uparrow \!y\subseteq U$. Since $\Uparrow_{s}\!y\in\sigma_{s}(L)$, $U\in\sigma_{s}(L)$.\\
$(3)\Rightarrow(1)$: $L$ is continuous iff $(L,\sigma(L))$ is a $C$-space iff $(L,\sigma_{s}(L))$ is a $C$-space iff $L$ is strongly continuous.
\end{proof}

The following example shows that a continuous domain need not be a strongly continuous domain.
\begin{example}\label{example3.12}(\cite{XZ}) Let $C=\{a_{1},a_{2},\cdots,a_{n},\cdots\}\cup\{\omega_{0}\}$ and $L=C\cup\{b\}\cup\{\omega_{1},\omega_{2},\cdots,\omega_{n},\cdots\}$ with the order generated by
\begin{enumerate}
\item[\rm{(a)}] $a_{1}<a_{2}<\cdots<a_{n}<a_{n+1}<\cdots$;
\item[\rm{(b)}] $a_{n}<\omega_{0}$ for all $n\in\mathbb{N}$;
\item[\rm{(c)}] $b<\omega_{n}$ and $a_{m}<\omega_{n}$ for all $n,m\in\mathbb{N}$ with $m\leq n$ (Fig. \ref{Fig.2}).
\end{enumerate}
\end{example}

\begin{figure}[H]
  \centering
  \includegraphics[width=0.5\textwidth]{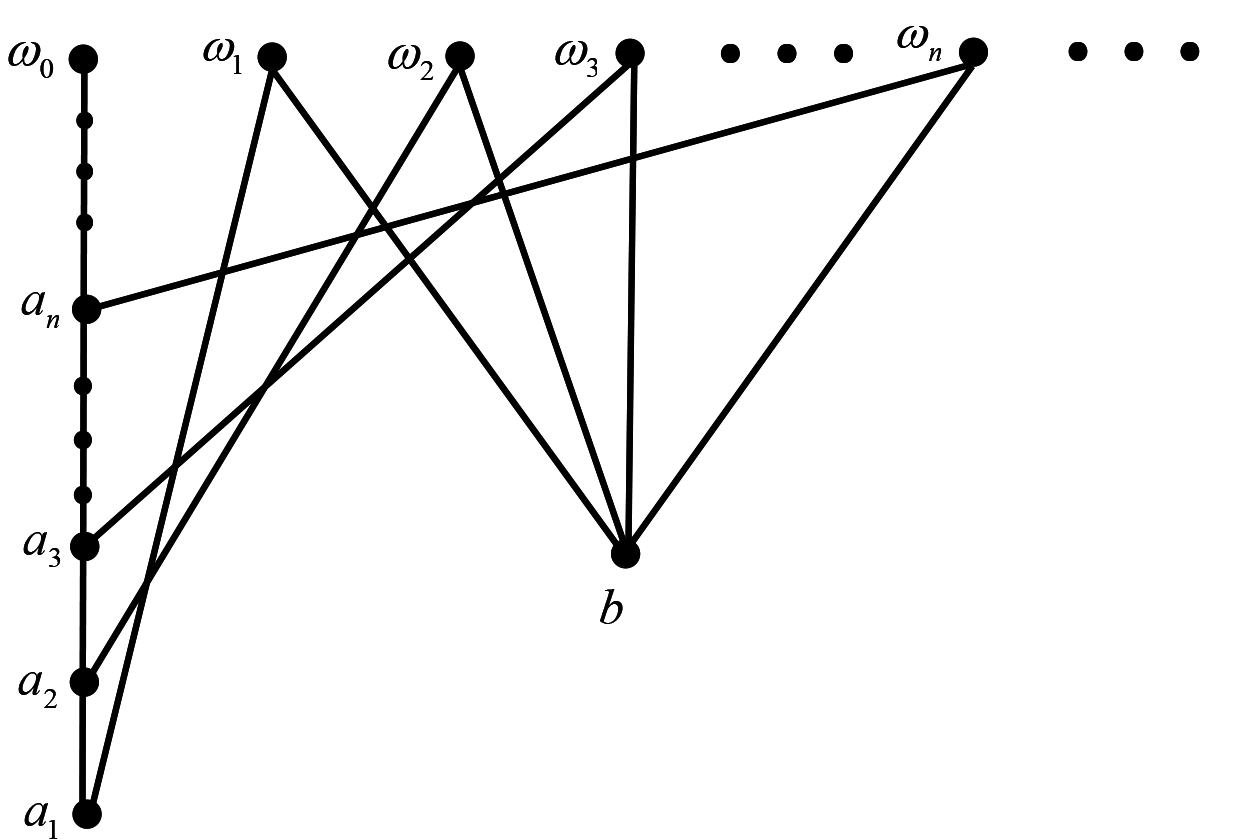}\\
  \caption{The poset $L$.}
  \label{Fig.2}
\end{figure}

Then $L$ is a dcpo and $D$ is a directed set of $L$ iff $D\subseteq C$ or $D$ has a largest element, and hence $x\ll x$ for all $x\in L\setminus\{\omega_{0}\}$. Therefore $L$ is a continuous domain. However, $L$ is not a strongly continuous domain. In fact, choose $D=\{a_{n}:n\in\mathbb{N}\}$, then $\bigcap_{d\in D}\!\uparrow \!\!d\ \!\cap\!\uparrow \!\!b\subseteq\ \!\uparrow \!\!\omega_{1}=\{\omega_{1}\}$, but $\uparrow \!\!a_{n}~\cap\uparrow \!\!b=\{\omega_{n},\omega_{n+1},\cdots\}\nsubseteq\ \!\uparrow\!\!\omega_{1}$ for each $n\in\mathbb{N}$. Hence $\omega_{1}\notin\ \!\Downarrow_{s}\!\!\omega_{1}$. Since $\Downarrow_{s}\!\omega_{1}\subseteq\ \!\downarrow\!\omega_{1}=\{\omega_{1},a_{1},b\}$, $\Downarrow_{s}\!\omega_{1}\subseteq\{a_{1},b\}$, we have $\omega_{1}\neq\bigvee\Downarrow_{s}\!\omega_{1}$.\\

According to the Theorem \ref{the3.11} and Example \ref{example3.12}, a strongly continuous domain is a continuous domain, but the converse is not true. Therefore strongly continuous domains lie strictly between hypercontinuous domains and continuous domains.

\section{Topologies on strongly continuous domains}

In this section, we study some properties of strongly continuous domains endowed with strong Scott topology and strong Lawson topology.

\begin{proposition} \label{pro-4.1}Let $L$ be a dcpo. Then we have the following conditions.
\begin{enumerate}
\item[\rm(1)] $\textrm{cl}_{\sigma_{s}(L)}\{x\}=\ \downarrow \!x$ for all $x\in L$;
\item[\rm(2)] $\sigma_{s}(L)$ is a $T_{0}$-topology;
\item[\rm(3)] If $A=\ \uparrow \!A$, then $A=\bigcap\{U\in\sigma_{s}(L):A\subseteq U\}$.
\end{enumerate}
\end{proposition}

\begin{proposition} \label{pro-4.2}Let $L$ be a dcpo. Consider the following two conditions:
\begin{enumerate}
\item[\rm(1)] $y\in \textrm{int}_{\sigma_{s}(L)}\uparrow \!x$;
\item[\rm(2)] $x\ll_{s}y$.
\end{enumerate}
Then $(1)\Rightarrow(2)$; if $L$ is strongly continuous, then $(2)\Rightarrow(1)$, and $(1)$ and $(2)$ are equivalent.
\end{proposition}

\begin{proof}
$(1)\Rightarrow(2)$: By the proof of Theorem \ref{The-3.5}.\\
$(2)\Rightarrow(1)$: Let $x\ll_{s}y$, then $y\in\ \!\Uparrow_{s}\!\!x\subseteq\ \!\uparrow \!\!x$. Since $L$ is strongly continuous, $\Uparrow_{s}\!x\in\sigma_{s}(L)$. Hence $y\in\ \!\Uparrow_{s}\!x\subseteq\textrm{int}_{\sigma_{s}(L)}\uparrow \!x$.
\end{proof}

\begin{proposition} \label{pro-4.3}Let $L$ be a strongly continuous domain. Then we have the following conditions.
\begin{enumerate}
\item[\rm(1)] If $U=\ \!\uparrow \!U$, then $U\in\sigma_{s}(L)$ iff there is a $u\in U$ such that $x\in\ \!\Uparrow_{s}\!u$ for all $x\in U$;
\item[\rm(2)] $\{\Uparrow_{s}\!u:u\in L\}$ form a basis for the strong Scott topology;
\item[\rm(3)] $\textrm{int}_{\sigma_{s}(L)}\uparrow \!x=\ \!\Uparrow_{s}\!x$ for all $x\in L$;
\item[\rm(4)] For any subset $X\subseteq L$, $\textrm{int}_{\sigma_{s}(L)}X=\bigcup\{\Uparrow_{s}\!u:\Uparrow_{s}\!u\subseteq X\}$.
\end{enumerate}
\end{proposition}
\begin{proof}
(1) For each $x\in U$, since $L$ is strongly continuous, $\Downarrow_{s}\!\!x$ is directed and $x=\bigvee\Downarrow_{s}\!x$. And hence $x=\bigvee\Downarrow_{s}\!x \in\sigma_{s}(L)\subseteq\sigma(L)$, thus $\Downarrow_{s}\!x\cap U\neq\emptyset$. Therefore there exists $u\in U$ such that $x\in\ \!\Uparrow_{s}\!u$. On the other hand, it is trivial.\\
(2) It follows directly from (1).\\
(3) It follows from Proposition \ref{pro-4.2}.\\
(4) On the one hand, obviously, $\bigcup\{\Uparrow_{s}\!\!u:\ \!\Uparrow_{s}\!\!u\subseteq X\}\subseteq X$ and $\bigcup\{\Uparrow_{s}\!\!u:\ \!\Uparrow_{s}\!\!u\subseteq X\}\in\sigma_{s}(L)$, then $\bigcup\{\Uparrow_{s}\!\!u:\ \!\Uparrow_{s}\!\!u\subseteq X\}\subseteq\textrm{int}_{\sigma_{s}(L)}X$. On the other hand, for each $y\in\textrm{int}_{\sigma_{s}(L)}X$, there exists $u\in\textrm{int}_{\sigma_{s}(L)}X$ such that $y\in\ \!\Uparrow_{s}\!u\subseteq\textrm{int}_{\sigma_{s}(L)}X$. Thus $\Uparrow_{s}\!u=\textrm{int}_{\sigma_{s}(L)}\uparrow \!u\subseteq\ \!\uparrow \!u\subseteq\textrm{int}_{\sigma_{s}(L)}X\subseteq X$. Therefore $\textrm{int}_{\sigma_{s}(L)}X\subseteq\bigcup\{\Uparrow_{s}\!u:\Uparrow_{s}\!u\subseteq X\}$.
\end{proof}

Now we recall the definitions of \emph{prime elements}, \emph{sober spaces}. An element $p$ in a poset $L$ is called \emph{prime} iff $p=1$ or $L\setminus\downarrow p$ is a filter. An element is \emph{co-prime} iff it is prime of $L^{\textrm{op}}$. The sets of all prime and co-prime elements of $L$ are denoted by $\textrm{PRIME}(L)$ and $\textrm{COPRIME}(L)$, respectively. A nonempty subset $A$ of a $T_{0}$ space $X$ is said to be \emph{irreducible} if for any $\{B,C\}\subseteq \mathcal{C}(X)$, $A\subseteq B\cup C$ implies $A\subseteq B$ or $A\subseteq C$. Obviously, $\textrm{cl}\{x\}$ is a irreducible closed set for all $x\in X$. A space is called \emph{sober}, if for any irreducible closed set $C$, there is a unique point $x\in X$ such that $C=\textrm{cl}\{x\}$.

\begin{lemma}\label{lem4.4}{\rm(\cite{GHK})} Let $X$ be a space and $A$ a subset of $X$. Then the following conditions are equivalent:
\begin{enumerate}
\item[\rm{(1)}] $A$ is a irreducible closed set of $X$;
\item[\rm{(2)}] $A$ is a co-prime in lattice $\mathcal{C}(X)$;
\item[\rm{(3)}] $X\setminus A$ is a prime in lattice $\mathcal{O}(X)$;
\item[\rm{(4)}] For any $U,V\in\mathcal{O}(X)$, if $U\cap A\neq\emptyset$ and $V\cap A\neq\emptyset$, then $U\cap V\cap A\neq\emptyset$.
\end{enumerate}
\end{lemma}

\begin{definition}
Let $L$ be a dcpo. The set of all strong Scott topology open filters of $L$ is denoted by $\textrm{SOFilt}(L)$, i.e., $U\in \textrm{SOFilt}(L)$ iff $U\in\sigma_{s}(L)$ and $U$ is a filter.
\end{definition}

\begin{proposition} \label{pro-4.6}Let $L$ be a dcpo and $U\in\sigma_{s}(L)$.
\begin{enumerate}
\item[\rm(1)] $U$ is a co-prime in $\sigma_{s}(L)$ iff $U\in \textrm{SOFilt}(L)$;
\item[\rm(2)] $L\setminus \downarrow \!a\in\textrm{PRIME}(\sigma_{s}(L))$ for all $a\in L$; For each $U\in\textrm{PRIME}(\sigma_{s}(L)), U\neq L$, if $L$ is strongly continuous, then there exists $a\in L$ such that $U=L\setminus \downarrow \!a$.
\end{enumerate}
\end{proposition}
\begin{proof}
(1):  On the one hand, suppose that $U\in\textrm{COPRIME}(\sigma_{s}(L))$, we only need to show that $U$ is a filter. For any $x,y\in U$, assume that $U\ \!\cap \downarrow \!x\ \!\cap \downarrow \!y=\emptyset$, then $U\subseteq(L\setminus \downarrow x)\cup(L\setminus \downarrow y)$. Since $U\in\textrm{COPRIME}(\sigma_{s}(L))$, $U\subseteq L\setminus \downarrow x$ or $U\subseteq L\setminus \downarrow y$, which is in contradiction with $U$ is an upper set. Therefore there is a $z\in U$ such that $z\leq x$ and $z\leq y$. On the other hand, suppose that $U\in\textrm{SOFilt}(L)$, if $V,W\in\sigma_{s}(L)$ and $U\nsubseteq V,U\nsubseteq W$, we only need check $U\nsubseteq V\cup W$. Since $U\nsubseteq V$ and $U\nsubseteq W$, then there exists $v\in U\setminus V$ and $w\in U\setminus W$. Note that $U$ is a filter, hence there exists $u\in U$ such that $u\leq v,w$, then $u\in U\setminus(V\cup W)$. Therefore $U\nsubseteq V\cup W$.\\
(2): By Proposition \ref{pro-4.1} and Lemma \ref{lem4.4}, $L\setminus\downarrow \!a=L\setminus \textrm{cl}_{\sigma_{s}(L)}\{a\}\in\textrm{PRIME}(\sigma_{s}(L))$ for all $a\in L$. Conversely, for each $U\in\textrm{PRIME}(\sigma_{s}(L))$ and $U\neq L$, let $A=L\setminus U$, then $A$ is an irreducible closed set in $(L,\sigma_{s}(L))$. We have to show that $A$ has a largest element $e$; since $A$ is a lower set, this will show that $A=\ \!\downarrow \!e$ as desired. Let $A^{\ast}=\bigcup\{\Downarrow_{s}a:a\in A\}=\ \!\Downarrow_{s}\!A\subseteq\ \!\downarrow \!A=A$,  we claim that $A^{\ast}$ is directed. For each $b,c\in A^{\ast}$, there exists $a_{b},a_{c}\in A$ such that $b\ll_{s}a_{b}$ and $c\ll_{s}a_{c}$, respectively. We first show that $\Uparrow_{s}\!b\ \!\cap\Uparrow_{s}\!c\cap A\neq\emptyset$, if not, then $\Uparrow_{s}\!b\ \!\cap\Uparrow_{s}\!c\subseteq U$, but $\Uparrow_{s}\!b, \Uparrow_{s}\!c\in\sigma_{s}(L)$ and $U\in\textrm{PRIME}(\sigma_{s}(L))$, then $\Uparrow_{s}\!b\subseteq U$ or $\Uparrow_{s}\!c\subseteq U$. But $\Uparrow_{s}b$ contains an $a_{b}\in A=L\setminus U$ which is impossible, similarly $\Uparrow_{s}\!c\subseteq U$ is impossible. Pick $a\in\ \!\Uparrow_{s}\!b\ \!\cap\Uparrow_{s}\!c\cap A$, since $\Uparrow_{s}\!b\ \!\cap\Uparrow_{s}\!c\in \sigma_{s}(L)$, there exists $d\in\ \!\Uparrow_{s}\!b\ \!\cap\Uparrow_{s}\!c$ such that $a\in\ \!\Uparrow_{s}\!d\subseteq\ \!\Uparrow_{s}\!b\ \!\cap\Uparrow_{s}\!c$. Thus $d\in A^{\ast}$ and $b,c\leq d$. Hence $A^{\ast}$ is directed and $\bigvee A^{\ast}$ exists. Note that $A^{\ast}\subseteq A$ and $A$ is a strong Scott topology closed set, then $A$ is also a Scott closed set. Let $e=\bigvee A^{\ast}$, then $e=\bigvee A^{\ast}\in A$. Now we show that $e$ is the largest element in $A$. For any $x\in A$, since $\Downarrow_{s}\!x\subseteq A^{\ast}$, $x=\bigvee\Downarrow_{s}\!x\leq\bigvee A^{\ast}=e$.
\end{proof}

\begin{corollary}
Let $L$ be a strongly continuous domain. Then $(L,\sigma_{s}(L))$ is a sober space.
\end{corollary}

\begin{proof}
By Lemma \ref{lem4.4} and Proposition \ref{pro-4.6}.
\end{proof}

\begin{corollary}
Let $L$ be a strongly continuous domain. Then $(L,\sigma_{s}(L))$ is a locally compact sober space. If $L$ has a smallest element, then $(L,\sigma_{s}(L))$ is compact.
\end{corollary}

\begin{proof}
Suppose that $x\in U\in\sigma_{s}(L)$, by Proposition \ref{pro-4.3}, there exists $y\in U$ such that $x\in\ \!\Uparrow_{s}\!y\subseteq\ \!\uparrow \!y\subseteq U$. Note that $\uparrow \!y$ is compact with respect to any topology whose open sets are upper sets, thus the assertion is proved. If $L$ has a smallest element, let $\{U_{i}:i\in I\}\subseteq\sigma_{s}(L)$ and $L\subseteq\bigcup_{i\in I}U_{i}$, then there exists $i\in I$ such that $0\in U_{i}$, hence $L=\ \!\uparrow \!0\subseteq U_{i}$. Therefore $(L,\sigma_{s}(L))$ is compact.
\end{proof}

\begin{theorem}\label{the4.9} Let $L$ be a dcpo. Then the following conditions are equivalent:
\begin{enumerate}
\item[\rm{(1)}] $L$ is strongly continuous;
\item[\rm{(2)}] For each $x\in L$, $\Uparrow_{s}\!x\in\sigma_{s}(L)$ and if $U\in\sigma_{s}(L)$, then $U=\bigcup\{\Uparrow_{s}\!u:u\in U\}$;
\item[\rm{(3)}] $\textrm{SOFilt}(L)$ is a basis of $\sigma_{s}(L)$ and $\sigma_{s}(L)$ is a continuous lattice;
\item[\rm{(4)}] $\sigma_{s}(L)$ has enough co-primes and is a continuous lattice;
\item[\rm{(5)}] $\sigma_{s}(L)$ is completely distributive;
\item[\rm{(6)}] Both $\sigma_{s}(L)$ and $\sigma_{s}(L)^{\textrm{op}}$ are continuous.
\end{enumerate}
If $L$ is a complete semilattice, then these conditions are equivalent to
\begin{enumerate}
\item[\rm{(7)}] For each point $x\in L$, $\Uparrow_{s}\!x\in\sigma_{s}(L)$ and $x=\textrm{sup}\{\textrm{inf}\ \!U:x\in U\in\sigma_{s}(L)\}$.
\end{enumerate}
\end{theorem}

\begin{proof}
$(1)\Leftrightarrow(5)$: By Proposition \ref{pro-2.4} and Definition \ref{def-3.1}.\\
$(1)\Rightarrow(2)$: By Proposition \ref{pro-4.3} $(1)$.\\
$(2)\Rightarrow(1)$: Let $x\in L$. If $u\ll_{s}x$, $v\ll_{s}x$, then there exists $w\in\ \!\Uparrow_{s}\!u\ \!\cap\Uparrow_{s}\!v$ such that $x\in\ \!\Uparrow_{s}\!w$. Thus $\Downarrow_{s}\!x$ is directed. Set $y=\bigvee\Downarrow_{s}\!x\leq x$, if $y<x$, then $L\setminus\downarrow \!y$ is a strong Scott topology open neighborhood of $x$. Hence there exists $z\in L\setminus\downarrow \!y$ such that $x\in\ \!\Uparrow_{s}\!z$, then $z\ll_{s}x$, and thus $z\leq\bigvee\Downarrow_{s}\!x=y$, which is a contradiction.\\
$(3)\Leftrightarrow(4)$: By Proposition \ref{pro-4.6} $(1)$.\\
$(4)\Leftrightarrow(5)\Leftrightarrow(6)$: By Theorem \ref{the-2.6}.\\
$(3)\Rightarrow(7)$: Suppose that $L$ is a complete semilattice, then for each $U\in\sigma_{s}(L)$ and each $x\in U$, $\textrm{inf}\ \!U$ exists. Let $y=\textrm{sup}\{\textrm{inf}\ \!U:x\in U\in\sigma_{s}(L)\}$. Then $y\leq x$. Assume that $x\nleq y$, then $x\in L\setminus\downarrow y\in\sigma_{s}(L)$. Since $\sigma_{s}(L)$ is a continuous lattice, there exists $V\in\sigma_{s}(L)$ such that $x\in V\ll L\setminus\downarrow y$. And by the conditions, there exists $U\in\textrm{SOFilt}(L)$ such that $x\in U\subseteq V\ll L\setminus\downarrow y$. Note that $\textrm{inf}\ \!U\leq y$, hence $L\setminus\downarrow y\subseteq L\setminus\downarrow\textrm{inf}\ \!U=L\setminus\bigcap\{\downarrow u:u\in U\}=\bigcup\{L\setminus\downarrow u:u\in U\}$. Since $U$ is a filter, $\{L\setminus\downarrow u:u\in U\}$ form a directed family of strong Scott topology open sets. Since $U\ll L\setminus\downarrow y$, there is a $u\in U$ with $U\subseteq L\setminus\downarrow u$, which is a contradiction. Thus $x\leq y$. Since $(2)\Leftrightarrow(3)$, we have $\Uparrow_{s}\!x\in\sigma_{s}(L)$ for all $x\in L$.\\
$(7)\Rightarrow(1)$: We only need to show that for each $U\in\sigma_{s}(L)$ and each $x\in U$, $\textrm{inf}\ \!U\ll_{s}x$. For any directed set $D\subseteq L$ and $z\in L$, if $\uparrow\!\bigvee \!D\ \!\cap\uparrow \!z\subseteq\ \!\uparrow \!x$, since $x\in U\in\sigma_{s}(L)$, there exists $V\in\sigma^{s}(L)$ such that $x\in V\subseteq U$, and hence $\uparrow\!\bigvee \!D\ \!\cap\uparrow \!z\subseteq\ \!\uparrow \!x\subseteq V$. Thus there is a $d\in D$ with $\uparrow \!d\ \!\cap\uparrow \!z\subseteq V$, then $\uparrow \!d\ \!\cap\uparrow \!z\subseteq V\subseteq U\subseteq\ \!\uparrow\textrm{inf}\ \!U$. Therefore $\textrm{inf}~U\ll_{s}x$.
\end{proof}

\begin{corollary}
If $L$ is a complete semilattice, then $L$ is hypercontinuous iff $x=\sup \{\inf U:x\in U\in\upsilon(L)\}$ for all $x\in L$.
\end{corollary}
\begin{proof}
By Theorem \ref{the3.9} and Theorem \ref{the4.9}.
\end{proof}

Now we introduce the notation of strong Lawson topology and discuss the properties of strongly continuous domains endowed with the strong Lawson topology.

\begin{definition}
Let $L$ be a dcpo. The common refinement $\sigma_{s}(L)\vee\omega(L)$ of the strong Scott topology and the lower topology is called the \emph{strong Lawson topology} and is denoted by $\lambda_{s}(L)$.
\end{definition}

\begin{lemma}{\rm(\cite{GHK})}\label{lem4.12} Let $L$ be a complete semilattice. Then $\lambda(L)$ is a compact $T_{1}$ topology.
\end{lemma}

\begin{theorem}\label{the4.13} Let $L$ be a complete semilattice. Then $\lambda_{s}(L)$ is a compact $T_{1}$ topology.
\end{theorem}
\begin{proof}
By Lemma \ref{lem4.12}, $(L,\lambda_{s}(L))$ is compact. For each $x\in L$, $\{x\}=\ \!\downarrow \!x\ \!\cap\uparrow \!x$. Since $\downarrow \!x$ is a strong Scott topology closed set and $\uparrow \!x$ is a lower topology closed set, $\{x\}$ is a strong Lawson topology closed set. Thus $(L,\lambda_{s}(L))$ is a $T_{1}$ space.
\end{proof}

\begin{theorem}\label{the4.14} Let $L$ be a strongly continuous domain. Then $(L,\lambda_{s}(L))$ is a $T_{2}$ space.
\end{theorem}
\begin{proof}
Let $x,y\in L$, and $x\neq y$, assume that $x\nleq y$, then there exists $u\ll_{s}x$, such that $u\nleq y$. Thus $x\in\ \!\Uparrow_{s}\!u\in\sigma_{s}(L)$, $y\in L\setminus\uparrow \!u\in\omega(L)$, and $\Uparrow_{s}\!u\cap(L\setminus\uparrow \!u)=\emptyset$. Therefore $(L,\lambda_{s}(L))$ is a $T_{2}$ space.
\end{proof}

\begin{corollary}
Let $L$ is a strongly continuous complete semilattice. Then $\lambda_{s}(L)$ is compact and Hausdorff.
\end{corollary}
\begin{proof}
By Theorem \ref{the4.13} and Theorem \ref{the4.14}.
\end{proof}

Recall that the definition of Scott-continuous functions. For a function $f$ from a dcpo $P$ into a dcpo $Q$, if $f$ is continuous with respect to the Scott topologies, that is, $f^{-1}(U)\in\sigma(P)$ for all $U\in\sigma(Q)$, then $f$ is called a \emph{Scott-continuous function}.\\
\indent A function $f$ is called \emph{strongly Scott-continuous} if it is continuous with respect to the strong Scott topologies. Next we discuss the properties of strongly Scott-continuous functions.

\begin{proposition} Let $P$, $Q$ be dcpos and $f: P\rightarrow Q$. Consider the following conditions:
\begin{enumerate}
\item[\rm(1)] $f: (P,\sigma_{s}(P))\rightarrow (Q,\sigma_{s}(Q))$ is continuous;
\item[\rm(2)] $f^{-1}(U)\in\sigma_s(P)$ for all $U\in\sigma^{s}(Q)$;
\item[\rm(3)] $f^{-1}(U)\in\sigma^{s}(P)$ for all $U\in\sigma^{s}(Q)$;
\item[\rm(4)] For each directed set $D\subseteq P$ and $x\in P$, $\uparrow \!f(\bigcap_{d\in D}\uparrow \!d \ \!\cap\uparrow \!x)=\bigcap_{d\in D}\uparrow \!f(d) \ \!\cap\uparrow \!f(x)$.
\end{enumerate}
Then $(4)\Rightarrow (3)\Rightarrow (2)\Leftrightarrow (1)$; if both $P$ and $Q$ are sup semilattices and $f$ preserves finite sups, then $(1)\Rightarrow (4)$, and the four conditions are equivalent.
\end{proposition}
\begin{proof}
$(1)\Leftrightarrow (2)$: Trivial.\\
$(4)\Rightarrow (3)$: First we show that $f$ is order preserving. Suppose that $x\leq y$. Let $D=\{y\}$. Then $\uparrow \!f(\uparrow \!y\ \!\cap\uparrow \!x)=\ \!\uparrow \!f(y)\ \!\cap\uparrow \!f(x)$. Hence $\uparrow \!f(\uparrow \!y\ \!\cap\uparrow \!x)=\ \!\uparrow \!f(\uparrow \!y)$, then $f(y)\in\ \!\uparrow \!f(y)\ \!\cap\uparrow \!f(x)\subseteq\ \!\uparrow \!f(x)$. Thus $f(x)\leq f(y)$. For each $U\in\sigma^{s}(Q)$, since $f$ is order preserving, $f^{-1}(U)=\ \!\uparrow \!f^{-1}(U)$. Suppose directed set $D\subseteq P$ and $x\in P$ such that $\bigcap_{d\in D}\uparrow \!d\ \!\cap\uparrow \!x\subseteq f^{-1}(U)$, then $f(\bigcap_{d\in D}\uparrow \!d\ \!\cap\uparrow \!x)\subseteq U$, hence $\bigcap_{d\in D}\uparrow \!f(d)\ \!\cap\uparrow \!f(x)=\ \!\uparrow \!f(\bigcap_{d\in D}\uparrow \!d\ \!\cap\uparrow \!x)\subseteq U$. Since $U\in\sigma^{s}(Q)$, there is a $d\in D$ with $\uparrow \!f(d)\ \!\cap\uparrow \!f(x)\subseteq U$. Thus $\uparrow \!d\ \!\cap\uparrow \!x\subseteq f^{-1}(U)$. Therefore $f^{-1}(U)\in\sigma^{s}(P)$.\\
$(3)\Rightarrow (1)$: For each $U\in\sigma_{s}(Q)$, there exists $\{U_{i}:i\in I\}\subseteq\sigma^{s}(Q)$ such that $U=\bigcup_{i\in I}U_{i}$ by the definition of $\sigma_{s}(Q)$. Then $f^{-1}(U)=f^{-1}(\bigcup_{i\in I}U_{i})=\bigcup_{i\in I}f^{-1}(U_{i})$. By precondition $(2)$, $f^{-1}(U_{i})\in\sigma^{s}(P)$ for all $i\in I$. Thus $f^{-1}(U)\in\sigma_{s}(P)$.\\
$(1)\Rightarrow (4)$: Suppose that both $P$ and $Q$ are sup semilattices and $f$ preserves finite sups, then $\sigma_{s}(P)=\sigma(P), \sigma_{s}(Q)=\sigma(Q)$. For each directed set $D\subseteq P$ and $x\in P$, we have $\bigcap_{d\in D}\!\uparrow \!d\ \!\cap\uparrow \!x=\bigcap_{d\in D}\uparrow\!(d\vee x)=\ \!\uparrow\!\bigvee_{d\in D}(d\vee x)$, thus $\uparrow \!f(\bigcap_{d\in D}\!\uparrow \!d\ \!\cap\uparrow \!x)=\ \!\uparrow \!f(\bigvee_{d\in D}(d\vee x))=\ \!\uparrow \!\bigvee_{d\in D}f(d\vee x)=\ \!\uparrow \!\bigvee_{d\in D}f(d)\vee f(x)=\bigcap_{d\in D}\uparrow\!(f(d)\vee f(x))=\bigcap_{d\in D}\uparrow \!f(d)\ \!\cap\uparrow \!f(x)$.
\end{proof}

\begin{proposition}
Let $P$ be a strongly continuous domain. Both $f: (P,\sigma_{s}(P))\rightarrow (Q,\sigma_{s}(Q))$ and $g: (Q,\sigma_{s}(Q))\rightarrow (P,\sigma_{s}(P))$ are continuous functions, and $f\circ g=\mathbf{1}_{Q}$. Then $Q$ is a strongly continuous domain.
\end{proposition}

\begin{proof}
Let $U\in\sigma_{s}(Q)$ and $y\in U$. Then $g(y)\in f^{-1}(U)$. Since $f$ is continuous, $f^{-1}(U)\in\sigma_{s}(P)$. Since $P$ is a strongly continuous domain, there exists $u\in P$ such that $g(y)\in \textmd{int}_{\sigma_{s}(P)}\!\uparrow \!u\subseteq\ \uparrow \!u\subseteq f^{-1}(U)$, thus $y\in g^{-1}(\textmd{int}_{\sigma_{s}(P)}\!\uparrow \!u)\subseteq\ \uparrow \!f(u)\subseteq U$. Since $g$ is continuous, $g^{-1}(\textmd{int}_{\sigma_{s}(P)}\!\uparrow \!u)\in\sigma_{s}(Q)$. Hence $y\in \textmd{int}_{\sigma_{s}(Q)}\!\uparrow \!f(u)\subseteq\ \uparrow \!f(u)\subseteq U$. Therefore $Q$ is a strongly continuous domain.
\end{proof}

\end{document}